\documentclass[sn-mathphys,Numbered]{sn-jnl}

\usepackage[sort&compress]{natbib}
\usepackage{graphicx}%
\usepackage{multirow}%
\usepackage{amsmath,amssymb,amsfonts,listings,mathtools, bm}%
\usepackage{amsthm}%
\usepackage{mathrsfs}%
\usepackage[title]{appendix}%
\usepackage{xcolor}%
\usepackage{textcomp}%
\usepackage{manyfoot}%
\usepackage{booktabs}%
\usepackage{algorithm}%
\floatname{algorithm}{Algorithm}
\usepackage{algorithmic,float}
\usepackage{listings,geometry}%
\usepackage{caption}

\theoremstyle{thmstyleone}%
\newtheorem{theorem}{Theorem}

\theoremstyle{thmstyletwo}%
\newtheorem{lemma}{Lemma}%
\theoremstyle{thmstylethree}%
\newtheorem{corollary}{Corollary}

\usepackage{footmisc}   

\begin{document}


\title[Article Title]{A general alternating direction implicit iteration method for solving complex symmetric linear systems}


\author*[1]{\fnm{Juan} \sur{Zhang}}
\email{zhangjuan@xtu.edu.cn}
\author*[2]{\fnm{Wenlu} \sur{Xun}}

\affil*[1]{Key Laboratory of Intelligent Computing and Information Processing
of Ministry of Education, Hunan Key Laboratory for Computation and Simulation in Science and
Engineering, School of Mathematics and Computational Science, Xiangtan University, Xiangtan,
Hunan, China}

\affil*[2]{School of Mathematics and Computational Science, Xiangtan University, Xiangtan, Hunan,
China}

 
\abstract{We have introduced the generalized alternating direction implicit iteration (GADI) method for solving large sparse complex symmetric linear systems and proved its convergence properties. Additionally, some numerical results have demonstrated the effectiveness of this algorithm. Furthermore, as an application of the GADI method in solving complex symmetric linear systems, we utilized the flattening operator and Kronecker product properties to solve Lyapunov and Riccati equations with complex coefficients using the GADI method. In solving the Riccati equation, we combined inner and outer iterations, first simplifying the Riccati equation into a Lyapunov equation using the Newton method, and then applying the GADI method for solution. Finally, we provided convergence analysis of the method and corresponding numerical results.}




\maketitle

\section{Introduction}\label{sec1}

We consider the numerical solution of large sparse complex symmetric linear systems in the following form:
\begin{equation}
\label{eq1}
Ax=(W+iT)x=b,\ \ \ A\in\mathbb C^{n\times n},\ \ \ x,\ b\in\mathbb C^{n},
\end{equation}
where $i=\sqrt{-1}$ represents the imaginary part, and $W,T\in\mathbb R^{n\times n}$ are real symmetric matrices, with at least one of them is positive definite. Without loss of generality, in this paper, we assume that $W$ is symmetric and positive definite. This type of linear system commonly arises in various practical problems, particularly in scientific computing and engineering applications, such as diffuse optical tomography \cite{ref1}, structural dynamics \cite{ref2}, quantum mechanics \cite{ref3}, molecular scattering \cite{ref4}, algebraic eigenvalue problems [\cite{ref5}, and some time-dependent partial differential equation (PDE) solutions based on FFT \cite{ref6}.

In recent years, many scholars have developed iterative methods to effectively solve complex symmetric linear systems \eqref{eq1}. For example, Bai et al. \cite{ref7} introduced the use of Hermitian and skew-Hermitian splitting (HSS) methods to solve non-Hermitian positive definite linear systems. Subsequently, they applied the HSS method to preconditioned block linear systems and proposed the preconditioned HSS (PHSS) iterative method to improve the convergence speed of the HSS method \cite{ref8}. Furthermore, Bai and Golub \cite{ref8} proposed the accelerated HSS (AHSS) iterative method by utilizing parameter acceleration strategies, and numerical results indicated that the AHSS method outperformed the PHSS method. Subsequently, Bai et al. \cite{ref10} addressed the drawbacks of directly using the HSS method and proposed the modified HSS (MHSS) iterative method as follows:\\
\textbf{\ The MHSS iteration method}\quad Let $x^{(0)}\in\mathbb C^{n}$ be an initial guess and $\alpha$ be a given positive constant. For $k=0,1,2,\cdots$ until $\{x^{(k)}\}$ converges, compute $x^{(k+1)}$ according to the
following iteration scheme:
\begin{equation}\label{eq2}
\left\{
\begin{array}{ll}
(\alpha I+W)x^{(k+\frac{1}{2})}=(\alpha I-iT)x^{(k)}+b,\\
(\alpha I+T)x^{(k+1)}=(\alpha I+iW)x^{(k+\frac{1}{2})}-ib.
\end{array}
\right.
\end{equation}
In \cite{ref10}, it has shown that the MHSS iterative method converges for any positive constant $\alpha$ if $T$ is symmetric positive semi-definite. Clearly, the MHSS method only requires solving two linear subsystems with real symmetric positive definite coefficient matrices $\alpha I+W$ and $\alpha I+T$ at each step, making it more computationally efficient than the HSS method. Subsequently, Bai et al. proposed the preconditioned version of the MHSS method (PMHSS) for solving a class of complex symmetric linear systems in literature \cite{ref11}. The specific form of the PMHSS method is as follows:\\
\textbf{\ The PMHSS iteration method}\quad Let $x^{(0)}\in\mathbb C^{n}$ be an arbitrary initial value, $\alpha$ be a given positive constant, and $V\in\mathbb R^{n\times n}$ be a specified symmetric positive definite matrix. For $k=0,1,2,\cdots$ until $\{x^{(k)}\}$ converges, compute $x^{(k+1)}$ according to the
following sequence:
\begin{equation}\label{eq3}
\left\{
\begin{array}{ll}
(\alpha V+W)x^{(k+\frac{1}{2})}=(\alpha V-iT)x^{(k)}+b,\\
(\alpha V+T)x^{(k+1)}=(\alpha V+iW)x^{(k+\frac{1}{2})}-ib.
\end{array}
\right.
\end{equation}
In particular, when we choose $V=I$, the PMHSS iterative method evolves into the MHSS iterative method. Salkuyeh et al. \cite{ref12} proposed the Generalized Successive Over-Relaxation (GSOR) iterative method to solve this linear system. Additionally, Wang et al. \cite{ref13} introduced a combination method of real and imaginary parts (CRI), which is a special variant of the PMHSS method. Their theoretical analysis demonstrated that the spectral radius bound of the CRI iterative matrix is smaller than that of the PMHSS method, and the CRI method exhibits faster convergence than the PMHSS method. The specific form of the CRI method is as follows:\\
\textbf{\ The CRI iteration method}\quad Given an initial value $x^{(0)}\in\mathbb C^{n}$ and a positive constant $\alpha$. For $k=0,1,2,\cdots$ until $\{x^{(k)}\}$ converges, compute compute $x^{(k+1)}$ based
on the following update rule:
\begin{equation}\label{eq4}
\left\{
\begin{array}{ll}
(\alpha T+W)x^{(k+\frac{1}{2})}=(\alpha-i)Tx^{(k)}+b,\\
(\alpha W+T)x^{(k+1)}=(\alpha+i)Wx^{(k+\frac{1}{2})}-ib.
\end{array}
\right.
\end{equation}
Recently, Hezari et al. proposed the scale and split (SCSP) iterative method for solving equation \eqref{eq1} in \cite{ref14}, where they utilized the matrix splitting $$A=\frac{1}{\alpha-i}[(\alpha W+T)-i(W-\alpha T)].$$Furthermore, Salkuyeh in \cite{ref15} leveraged the idea of the SCSP iterative method and proposed the Two-Step Scale and Split (TSCSP) method for solving complex symmetric linear systems.\\
\textbf{\ The TSCSP iteration method}\quad Given an initial value $x^{(0)}\in\mathbb C^{n}$ and a positive constant $\alpha$. For $k=0,1,2,\cdots$ until $\{x^{(k)}\}$ converges, compute $x^{(k+1)}$ based
on the following update rule:
\begin{equation}\label{eq5}
\left\{
\begin{array}{ll}
(\alpha W+T)x^{(k+\frac{1}{2})}=i(W-\alpha T)x^{(k)}+(\alpha-i)b,\\
(\alpha T+W)x^{(k+1)}=i(\alpha W-T)x^{(k+\frac{1}{2})}+(1-i\alpha)b.
\end{array}
\right.
\end{equation}
The theoretical analysis in \cite{ref15} indicates that if the matrices $W$ and $T$ are symmetric positive definite, the TSCSP iterative method will unconditionally converge, and numerical results demonstrate that the TSCSP iterative method outperforms the GSOR and SCSP iterative methods. In addition, Salkuyeh and Siahkolaei derived the two-parameter TSCSP (TTSCSP) iterative method using a parameter strategy in \cite{ref17}. Recently, Cao and Wu proposed the minimal residual HSS (MRHSS) iterative method from a different perspective in \cite{ref18}, aiming to enhance the efficiency of the HSS iterative method by incorporating the technique of minimal residual in the HSS iteration.

In this paper, based on the GADI iterative method proposed in literature \cite{ref19}, we explore its application in solving large sparse real linear systems. We present the GADI iterative format for complex linear system \eqref{eq1} and analyze its convergence. Numerical results demonstrate that the GADI iterative method outperforms the MHSS, PMHSS, CRI, and TSCSP iterative methods. Furthermore, we apply this method to solve Lyapunov and Riccati equations with complex coefficients.

The structure of the remaining sections of this paper is as follows: In section \ref{sec:solve complex symmetric linear systems}, we present the GADI iterative format for solving complex symmetric linear systems \eqref{eq1} and investigate the convergence of this new method, along with providing corresponding numerical results. In section \ref{sec:solve Lyapunov equation}, we apply the GADI method to solve Lyapunov equations with complex coefficients. In section \ref{sec:solve Riccati equation}, we combine the Newton method and GADI method to solve Riccati equations with complex coefficients. Finally, in section \ref{sec:Conclusions}, we present some conclusions and remarks to conclude this paper.

In this paper, we will utilize the following notations: $\mathbb C^{n\times m}$ denotes the set of all $n\times m$ complex matrices, and $\mathbb R^{n\times m}$ denotes the set of all $n\times m$ real matrices. For a matrix $A$, we use $A^{T}$ and $A^{-1}$ to represent the transpose and inverse of $A$, and $\lambda(A)$ denotes the eigenvalues of $A$. Additionally, $\rho(A)$ represents the spectrum of $A$, and null$(A)$ denotes the null space of $A$. When $A$ and $B$ are symmetric, $A\succeq B$ indicates that $A-B$ is positive semi-definite, and $\|A\|_{2}$ denotes the 2-norm of $A$. Re$(A)$ and Im$(A)$ represent the real and imaginary parts of the eigenvalues of $A$, respectively, while $A\otimes I$ denotes the Kronecker product of $A$ and $I$.
\section{The GADI method for solving complex symmetric linear systems}
\label{sec:solve complex symmetric linear systems}
\subsection{GADI for solving complex symmetric linear systems}
\qquad Firstly,  the GADI method proposed in \cite{ref19} is used to solve the real linear equations $Ax=b$.  Its iterative format is as follows:
\begin{equation}\label{eq6}
\left\{
\begin{array}{ll}
(\alpha I+W)x^{(k+\frac{1}{2})}=(\alpha I-T)x^{(k)}+b,\\
(\alpha I+T)x^{(k+1)}=(T-(1-\omega)\alpha I)x^{(k)}+(2-\omega)\alpha x^{(k+\frac{1}{2})},
\end{array}
\right.
\end{equation}
where $A=W+T\in\mathbb R^{n\times n}$, and $\alpha>0,\ \ 0\leq\omega<2$.

We consider applying the GADI iterative method to solve the complex linear system \eqref{eq1}, assuming that $A$ is a non-singular complex symmetric matrix, where $A=W+iT, W\in\mathbb R^{n\times n}$ is a symmetric positive definite matrix and $T\in\mathbb R^{n\times n}$ is a symmetric positive semi-definite matrix. We add parameters to obtain$$\alpha x+Wx=\alpha x-iTx+b,$$
$$\alpha x+iTx=iTx-(1-\omega)\alpha x+(1-\omega)\alpha x+\alpha x=(iT-(1-\omega)\alpha I)x+(2-\omega)\alpha x,$$
where $\alpha>0,\ \ 0\leq\omega<2$.\ Taking the initial value as $x_{0}$=0, the GADI iterative format is as follows:
\begin{equation}\label{eq7}
\left\{
\begin{array}{ll}
(\alpha I+W)x^{(k+\frac{1}{2})}=(\alpha I-iT)x^{(k)}+b,\\
(\alpha I+iT)x^{(k+1)}=(iT-(1-\omega)\alpha I)x^{(k)}+(2-\omega)\alpha x^{(k+\frac{1}{2})}.
\end{array}
\right.
\end{equation}

Next, we present the algorithm flow as follows:
\floatname{algorithm}{Algorithm}
\renewcommand{\algorithmicrequire}{\textbf{Input:}}
\renewcommand{\algorithmicensure}{\textbf{Output:}}
\begin{algorithm}
  \caption{GADI iteration for solving complex symmetric linear systems \eqref{eq1}}
  \label{alg1}
  \begin{algorithmic}[1]
    \REQUIRE Matrix $W,T$, where $A=W+iT$, parameters $\alpha$ and $\omega$ are given positive constants, error limits $\varepsilon$, $\eta$, and $\tau$. $x^{(0)}\in\mathbb C^{n}$ can be any initial vector, and $b$ is the right-hand vector of \eqref{eq1};
    \ENSURE $x^{(k+1)}$ such that $x^{(k+1)}\approx x$, where $x$ is the solution to the complex linear system $Ax=b$.
    \STATE $x^{(0)}=0;$ calculate the residual vector $r^{(0)}=b-Ax^{(0)}$  and set $k=0$;
    \WHILE{$\frac{\|b-Ax^{(k)}\|}{\|b\|}=\frac{\|r^{(k)}\|}{\|b\|}\geq\varepsilon$}
    \STATE Calculate $r^{(k)}=b-Ax^{(k)}$, and let $\overline{r}^{(k)}=(\alpha I-iT)x^{(k)}+b$;
    \STATE Using the CG method to solve $$(\alpha I+W)x^{(k+\frac{1}{2})}=(\alpha I-iT)x^{(k)}+b,$$ where the approximate solution $x^{(k+\frac{1}{2})}$ satisfies $$\|\overline{r}^{(k)}-(\alpha I+W)x^{(k+\frac{1}{2})}\|\leq\eta\|\overline{r}^{(k)}\|,$$
    \STATE Calculate $r^{(k+\frac{1}{2})}=b-Ax^{(k+\frac{1}{2})}$, and let $$\overline{r}^{(k+\frac{1}{2})}=(iT-(1-\omega)\alpha I)x^{(k)}+(2-\omega)\alpha x^{(k+\frac{1}{2})},$$
    \STATE Using the CG method to solve $$(\alpha I+iT)x^{(k+1)}=(iT-(1-\omega)\alpha I)x^{(k)}+(2-\omega)\alpha x^{(k+\frac{1}{2})},$$ where the approximate solution $x^{(k+1)}$ satisfies  $$\|\overline{r}^{(k+\frac{1}{2})}-(\alpha I+iT)x^{(k+1)}\|\leq\tau\|\overline{r}^{(k+\frac{1}{2})}\|,$$
    \STATE Let $k=k+1;$
    \ENDWHILE
  \end{algorithmic}
\end{algorithm}\\
\subsection{Convergence analysis of GADI for solving complex symmetric linear systems}
\begin{lemma}\label{lemma1}\cite{ref16}
Let $W,\ T\in\mathbb R^{n\times n}$ be symmetric positive semi-definite matrices, and let $M=W+T$. Then, we have null$(W)\cap$null$(T)$=null$(M)$.
\end{lemma}
\begin{corollary}\label{corollary1}
Let $W,\ T\in\mathbb R^{n\times n}$ be symmetric positive semi-definite matrices. Let $A=W+iT$ and $M=W+T$. If $A$ is non-singular, then $M$ is a symmetric
positive definite matrix.
\end{corollary}
\begin{proof}
Since $A=W+iT$ is non-singular if and only if null$(W)\cap$null$(T)=\{0\}$. Under the assumption of this corollary, by Lemma \ref{lemma1}, we have null$(W)\cap$null$(T)$=null$(M)$. Therefore, $A$ is non-singular if and only if null$(M)=\{0\}$. Since both $W$ and $T$ are positive semi-definite, $M$ is also positive semi-definite, so null$(M)=\{0\}$ implies that $M$ is positive definite.
\end{proof}
\begin{lemma}\label{lemma2}\cite{ref16}
Let $W,T\in\mathbb R^{n\times n}$ be symmetric positive sem-idefinite matrices that satisfy null$(W)\cap$ null$(T)=\{0\}$. Then there exists a non-singular matrix $P\in\mathbb R^{n\times n}$ such that
\begin{equation}\label{eq8}
W=P^{T}\Lambda P,\ \ \ T=P^{T}\widetilde{\Lambda}P.
\end{equation}
Here, $\Lambda=\text{diag}(\lambda_{1},\lambda_{2},\cdots,\lambda_{n}),\ \widetilde{\Lambda}=\text{diag}(\widetilde{\lambda}_{1},\widetilde{\lambda}_{2},\cdots,\widetilde{\lambda}_{n})$, and $\lambda_{i},\ \widetilde{\lambda}_{i}$ satisfy
\begin{equation}\label{eq9}
\lambda_{i}+\widetilde{\lambda}_{i}=1,\ \ \lambda_{i}\geq0,\ \ \widetilde{\lambda}_{i}\geq0,\ (i=1,2,\cdots,n).
\end{equation}
\end{lemma}
\begin{theorem}\label{theorem1}
Let $W\in\mathbb R^{n\times n}$ be a symmetric positive definite matrix, $T\in\mathbb R^{n\times n}$ be a symmetric positive semi-definite matrix, $A=W+iT\in\mathbb C^{n\times n}$ be a non-singular matrix, and $\alpha$ be a given positive constant. Then the following HSS iteration method
\begin{equation}\label{eq10}
\left\{
\begin{array}{ll}
\alpha I+W)x^{(k+\frac{1}{2})}=(\alpha I-iT)x^{(k)}+b,\\
(\alpha I+iT)x^{(k+1)}=(\alpha I-W)x^{(k+\frac{1}{2})}+b,
\end{array}
\right.
\end{equation}
with iteration matrix $T(\alpha)$ having the spectral radius $\rho(T(\alpha))$, is bounded by the following inequality 
$$\sigma(\alpha)=\max_{\lambda_{i}\in sp(W)}\frac{|\alpha-\lambda_{i}|}{|\alpha+\lambda_{i}|},$$
where $T(\alpha)=(\alpha I+iT)^{-1}(\alpha I-W)(\alpha I+W)^{-1}(\alpha I-iT)$, and $sp(W)$ is the spectrum of the matrix $W$. Therefore, we have the following conclusion:
$$\rho(T(\alpha))\leq\sigma(\alpha)<1,\ \ \ \forall\alpha>0,$$
which implies that the HSS iteration \eqref{eq10} converges to the unique exact solution $x\in\mathbb C^{n}$ of the equation \eqref{eq1}.
\end{theorem}
\begin{proof}Since $T(\alpha)$ is similar to $\widehat{T}(\alpha)$, which can be expressed as
$$\widehat{T}(\alpha)=(\alpha I-W)(\alpha I+W)^{-1}(\alpha I-iT)(\alpha I+iT)^{-1},$$
we have
\begin{align*}
\rho(T(\alpha))&\leq\|(\alpha I-W)(\alpha I+W)^{-1}(\alpha I-iT)(\alpha I+iT)^{-1}\|_{2}\\
&\leq\|(\alpha I-W)(\alpha I+W)^{-1}\|_{2}\|(\alpha I-iT)(\alpha I+iT)^{-1}\|_{2}.
\end{align*}
Moreover, since $W\in\mathbb R^{n\times n}$ is symmetric positive definite and $T\in\mathbb R^{n\times n}$ is symmetric positive semi-definite, and null$(W)\cap$ null$(T)=\{0\}$, by Lemma \ref{lemma2}, there exists a non-singular matrix $P$ that satisfies \eqref{eq8} and \eqref{eq9}. Therefore, we have
\begin{align*}
\rho(T(\alpha))&\leq\|P^{T}(\alpha I-\Lambda)(\alpha I+\Lambda)^{-1}(\alpha I-i\widetilde{\Lambda})(\alpha I+i\widetilde{\Lambda})^{-1}P^{-T}\|_{2}\\
&\leq\|(\alpha I-\Lambda)(\alpha I+\Lambda)^{-1}\|_{2}\|(\alpha I-i\widetilde{\Lambda})(\alpha I+i\widetilde{\Lambda})^{-1}\|_{2}\\
&=\max_{\lambda_{i}\in sp(W)}\frac{|\alpha-\lambda_{i}|}{|\alpha+\lambda_{i}|}\max_{\widetilde{\lambda}_{i}\in sp(T)}\frac{|\alpha-i\widetilde{\lambda}_{i}|}{|\alpha+i\widetilde{\lambda}_{i}|}\\
&=\max_{\lambda_{i}\in sp(W)}\frac{|\alpha-\lambda_{i}|}{|\alpha+\lambda_{i}|}\max_{\widetilde{\lambda}_{i}\in sp(T)}\frac{\sqrt{\alpha^{2}+\widetilde{\lambda}^{2}_{i}}}{\sqrt{\alpha^{2}+\widetilde{\lambda}^{2}_{i}}}\\
&=\max_{\lambda_{i}\in sp(W)}\frac{|\alpha-\lambda_{i}|}{|\alpha+\lambda_{i}|},
\end{align*}
where $sp(W)$ and $sp(T)$ denote the spectra of matrice $W$ and $T$, respectively. Since $\lambda_{i}>0\ (i=1,2,\cdots,n)$ and $\alpha>0$, we have
$$\rho(T(\alpha))\leq\max_{\lambda_{i}\in sp(W)}\frac{|\alpha-\lambda_{i}|}{|\alpha+\lambda_{i}|}=\sigma(\alpha)<1.$$
\end{proof}
If the upper and lower bounds of the eigenvalues of the symmetric positive definite matrix $W$ are known, we can obtain the optimal parameter $\widetilde{\alpha}$ of $\sigma(\alpha)$ as stated in the following corollary.\\

\begin{corollary}\label{corollary2}
Assuming that all conditions in Theorem \ref{theorem1} are satisfied. Let $\gamma_{min}$ and $\gamma_{max}$ be the minimum and maximum eigenvalues of the symmetric positive definite matrix $W$, respectively. Then, we have $$\widetilde{\alpha}\equiv arg\min_{\alpha}\{\max_{\gamma_{min}\leq\lambda\leq\gamma_{max}}\frac{|\alpha-\lambda|}{|\alpha+\lambda|}\}=\sqrt{\gamma_{min} \gamma_{max}}\ ,$$ and $$\sigma(\widetilde{\alpha})=\frac{\sqrt{\gamma_{max}}-\sqrt{\gamma_{min}}}{\sqrt{\gamma_{max}}+\sqrt{\gamma_{min}}}=\frac{\sqrt{\kappa(W)}-1}{\sqrt{\kappa(W)}+1},$$ where $\kappa(W)$ is the condition number of matrix $W$.
\end{corollary}
\begin{proof}
The equation $$\sigma(\alpha)=\max\{\frac{|\alpha-\gamma_{min}|}{|\alpha+\gamma_{min}|},\ \frac{|\alpha-\gamma_{max}|}{|\alpha+\gamma_{max}|}\},$$
shows that to compute the optimal $\alpha$ such that the convergence factor $\rho(T(\alpha))$ of the HSS iteration is minimized, we can minimize the upper bound $\sigma(\alpha)$ of $\rho(T(\alpha))$. If $\widetilde{\alpha}$ is the parameter that minimizes $\sigma(\alpha)$, then it satisfies $\widetilde{\alpha}-\gamma_{min}>0,\ \widetilde{\alpha}-\gamma_{max}<0$, and
$$\frac{\alpha-\gamma_{min}}{\alpha+\gamma_{min}}= \frac{\alpha-\gamma_{max}}{\alpha+\gamma_{max}}.$$
Thus, we have $\widetilde{\alpha}=\sqrt{\gamma_{min}\gamma_{max}}$.

Obviously, $\widetilde{\alpha}$ is generally different from the optimal parameter $$\alpha_{\ast}=arg\min_{\alpha}\rho(T(\alpha)),$$ and it satisfies $$\rho(T(\alpha_{\ast}))\leq\sigma(\widetilde{\alpha}).$$\\
\end{proof}
\begin{theorem}\label{theorem2}
Let $W\in\mathbb R^{n\times n}$ be a symmetric positive definite matrix, $T\in\mathbb R^{n\times n}$ be a symmetric positive semi-definite matrix, and $A=W+iT\in\mathbb C^{n\times n}$ be a non-singular matrix. If the parameter $\alpha>0$ and $0\leq\omega<2$, then for any $k=0,1,2,\cdots$, the iteration sequence $\{x^{k}\}$ generated by the GADI iteration method defined by \eqref{eq7} converges to the unique exact solution $x\in\mathbb C^{n}$ of equation \eqref{eq1}.
\end{theorem}
\begin{proof}
From \eqref{eq7}, we can obtain $$x^{(k+\frac{1}{2})}=(\alpha I+W)^{-1}[(\alpha I-iT)x^{(k)}+b],$$
\begin{align*}
x^{(k+1)}&=(\alpha I+iT)^{-1}[(iT-(1-\omega)\alpha I)x^{(k)}+(2-\omega)\alpha x^{(k+\frac{1}{2})}]\\
&=(\alpha I+iT)^{-1}[(iT-(1-\omega)\alpha I)+(2-\omega)\alpha(\alpha I+W)^{-1}(\alpha I-iT)]x^{(k)}\\
&+(2-\omega)\alpha(\alpha I+iT)^{-1}(\alpha I+W)^{-1}b\\
&=(\alpha I+iT)^{-1}(\alpha I+W)^{-1}[\alpha^{2}I+iWT-(1-\omega)\alpha A]x^{(k)}\\
&+(2-\omega)\alpha(\alpha I+iT)^{-1}(\alpha I+W)^{-1}b.\\
\end{align*}
Simplifying the above equations into matrix form, we have $$x^{(k+1)}=M(\alpha,\omega)x^{(k)}+N(\alpha,\omega),$$ where$$M(\alpha,\omega)=(\alpha I+iT)^{-1}(\alpha I+W)^{-1}[\alpha^{2}I+iWT-(1-\omega)\alpha A],$$
$$N(\alpha,\omega)=(2-\omega)\alpha(\alpha I+iT)^{-1}(\alpha I+W)^{-1}b.$$
Furthermore, we have
\begin{align*}
2M(\alpha,\omega)&=2(\alpha I+iT)^{-1}(\alpha I+W)^{-1}[\alpha^{2}I+iWT-(1-\omega)\alpha A]\\
&=(2-\omega)(\alpha I+iT)^{-1}(\alpha I+W)^{-1}(\alpha I-W)(\alpha I-iT)\\
&+\omega(\alpha I+iT)^{-1}(\alpha I+W)^{-1}(\alpha I+W)(\alpha I+iT)\\
&=(2-\omega)T(\alpha)+\omega I,
\end{align*}
where $(\alpha I+W)^{-1}(\alpha I-W)=(\alpha I-W)(\alpha I+W)^{-1}$, and $T(\alpha)$ is the iterative matrix of the HSS method in equation \eqref{eq10}. Therefore, we have $$M(\alpha,\omega)=\frac{1}{2}[(2-\omega)T(\alpha)+\omega I].$$
Furthermore, we have $$\lambda_{i}(M(\alpha,\omega))=\frac{1}{2}[(2-\omega)\lambda_{i}(T(\alpha))+\omega]\ (i=1,2,\cdots n),$$
and $$\rho(M(\alpha,\omega))\leq\frac{1}{2}[(2-\omega)\rho(T(\alpha))+\omega].$$
According to the conclusion in Theorem \ref{theorem1} that $\rho(T(\alpha))<1$ and $0\leq\omega<2$, thus we can get
$$\rho(M(\alpha,\omega))\leq\frac{1}{2}[(2-\omega)\rho(T(\alpha))+\omega]<1.$$
Therefore, we can conclude that the iterative formula \eqref{eq7} is convergent.
\end{proof}
\subsection{Complexity analysis}
\qquad When considering the iterative format in equation \eqref{eq7} and the complexity of Algorithm \ref{alg1}, it is important to note that at each step of GADI iteration, two linear subsystems need to be solved, with coefficient matrices $\alpha I+W$ and $\alpha I+iT$, both of which are symmetric positive definite matrices. The matrix $\alpha I+W$ can be efficiently handled using real arithmetic, allowing for effective calculation of the first half step of the iteration $x^{(k+\frac{1}{2})}$ through Cholesky decomposition or inexact conjugate gradient method. However, the matrix $\alpha I+iT$ while also symmetric positive definite, is complex and non-Hermitian, requiring a more complex algorithm for accurate calculation in the second half step of the iteration $x^{(k+1)}$. A potential difficulty of GADI iteration \eqref{eq7} lies in the fact that even though there are ways to construct matrix $T$ to make the linear system involving $\alpha I+iT$ easy to solve, the shifted subsystems of linear equations still need to be solved at each iteration step. At present, there is also evidence suggesting that the accuracy of iteration \eqref{eq7} may decrease during internal solving, which could yield significant cost savings, particularly when taking advantage of the sparsity of the coefficient matrix of the linear system. We define the residual norm as follows: equation$$\|r^{(k)}\|_{2}=\|b-Ax^{(k)}\|_{2}.$$

In practical computations, iterative schemes typically start from the zero vector, and due to the presence of errors, the computational time required for iteration may become unacceptable when $n$ is large. Therefore, it is necessary to set a stopping criterion, such as terminating when the current iteration satisfies ERR $\leq10^{-6}$, where $$\text{RES}=\frac{\|b-Ax^{(k)}\|_{2}}{\|b\|_{2}}.$$
In Algorithm \ref{alg1}, the major computations required for each iteration include matrix-vector multiplication, inner product of vectors, and linear combinations of vectors, resulting in a computational complexity of $O(n^{2})$ for solving these two linear systems.
\subsection{Numerical experiments}
\qquad\textbf{Example 2.4.1}\cite{ref10}\ \ First, we consider the linear system \eqref{eq1} in the following form:
\begin{equation}\label{eq11}
[(K+\frac{3-\sqrt{3}}{\tau}I)+i(K+\frac{3+\sqrt{3}}{\tau}I)]x=b,
\end{equation}
where $\tau$ is the time step, $h=\frac{1}{m+1}$ is the grid size, $K$ is the five-point central difference matrix, and matrix $K\in\mathbb R^{n\times n}$ has a tensor product form $K=I\otimes V_{m}+V_{m}\otimes I$, with$$
V_{m}=\frac{1}{h^{2}}\left(
\begin{array}{ccccc}
    2   &  -1   &  0   &  \cdots   &  0\\
     -1 &  2   &  -1  &   \cdots   &  0\\
     \vdots   &  \ddots  &  \ddots  &   \ddots   &  \vdots\\
     0   &  \cdots   &  -1  &   2   &  -1\\
     0   & \cdots   &  0  &   -1  &   2\\
  \end{array}
\right)_{m\times m},$$
where $K$ is a $n\times n$ block diagonal matrix, with $n=m^{2}$. We take $$W=K+\frac{3-\sqrt{3}}{\tau}I,\ \ \ T=K+\frac{3+\sqrt{3}}{\tau}I,$$
and the $j$-th component of the right-hand vector $b$ is given by $$b_{j}=\frac{(1-i)j}{\tau(j+1)^{2}},\ \ j=1,2,\cdots,n.$$

In the algorithm implementation for solving this example, we consider $\tau=h$ and $\tau=500h$, with the initial value chosen as $x^{(0)}=0$, and terminate the iteration once the current iterate $x^{(k)}$ satisfies the condition
$$\frac{\|b-Ax^{(k)}\|_{2}}{\|b\|_{2}}\leq10^{-6}.$$

Firstly, with the same value of $\tau$, we compare the GADI method with previously proposed MHSS, PMHSS, CRI and TSCSP methods, as shown in Table \ref{tab1}. From this table, we can see that not only in terms of iteration steps (IT), but also in terms of iteration time (CPU), the GADI method is superior to the MHSS, PMHSS, CRI and TSCSP methods. Secondly, for different values of $\tau$, it can be observed that when $\tau$ is relatively small compared to $h$, the number of iteration steps is relatively small. In addition, as shown in Fig. \ref{fig1}, for $m=32(n=1024)$ and $\tau=h$, the residual of these methods decreases with the iteration steps, in which the residual of the GADI method decreases the fastest, indicating that it has the fastest convergence rate.
\begin{table}[!htbp]
\caption{Numerical results for Example 2.4.1\label{tab1}}
\begin{tabular}{llllll}
   \hline
    algorithm &$m^{2}$ & time step~($\tau$) & RES &IT &CPU  \\
   \hline
    MHSS&$8^{2}$(64)& $\tau=h$ & 7.8469e-6 & 31 & 0.059s  \\

    PMHSS& $8^{2}$(64)& $\tau=h$ & 8.76e-6 & 17 & 0.024s  \\

   CRI & $8^{2}$(64)& $\tau=h$ & 9.0879e-6& 17& 0.022s \\

   TSCSP & $8^{2}$(64)& $\tau=h$ & 7.5704e-6 &  11&  0.015s \\

    GADI&  $8^{2}$(64)& $\tau=h$ &7.6464e-6 & 5 &  0.0079s\\

    GADI& $8^{2}$(64)& $\tau=500h$ &4.0316e-6& 6& 0.024s\\
     \hline
    MHSS&$16^{2}$(256)& $\tau=h$ & 9.7199e-6& 30 & 0.5602s  \\

    PMHSS& $16^{2}$(256)& $\tau=h$ & 6.5147e-6 & 18 & 0.321s  \\

   CRI & $16^{2}$(256)& $\tau=h$ & 6.5564e-6& 18& 0.301s \\

   TSCSP & $16^{2}$(256)& $\tau=h$ & 6.4014e-6 &  12&  0.231s \\

    GADI&  $16^{2}$(256)& $\tau=h$ &2.2772e-6 & 6 &  0.127s\\

    GADI& $16^{2}$(256)& $\tau=500h$ &6.718e-6& 7& 0.4101s\\
   \hline
    MHSS&$24^{2}$(576)& $\tau=h$ &8.1036e-6 & 32 & 2.693s  \\

    PMHSS& $24^{2}$(576)& $\tau=h$ & 6.2366e-6 & 20 & 1.6725s  \\

   CRI & $24^{2}$(576)& $\tau=h$ & 6.0047e-6& 19& 1.576s\\

   TSCSP & $24^{2}$(576)& $\tau=h$ & 7.0699e-6 &  13& 1.1307s \\

    GADI&  $24^{2}$(576)& $\tau=h$ &9.5479e-6 & 6 &  0.675s\\

    GADI& $24^{2}$(576)& $\tau=500h$ &3.1447e-6& 7& 2.1395s\\
     \hline
    MHSS&$32^{2}$(1024)& $\tau=h$ & 7.8991e-6 & 33 & 12.212s  \\

    PMHSS& $32^{2}$(1024)& $\tau=h$ &7.9169e-6 & 22 & 8.2102s  \\

   CRI & $32^{2}$(1024)& $\tau=h$ & 7.8213e-6& 20& 7.6119s \\

   TSCSP & $32^{2}$(1024)& $\tau=h$ & 4.7614e-6 &  14&  5.499s \\

    GADI&  $32^{2}$(1024)& $\tau=h$ &5.7213e-6 & 5 &  2.4403s\\

    GADI& $32^{2}$(1024)& $\tau=500h$ &1.5449e-6& 5& 7.5684s\\
     \hline
    MHSS&$48^{2}$(2304)& $\tau=h$ & 9.4894e-6 & 44 & 258.74s  \\

    PMHSS& $48^{2}$(2304)& $\tau=h$ & 8.4842e-6 & 31 & 172.63s  \\

   CRI & $48^{2}$(2304)& $\tau=h$ & 9.7241e-6& 22& 112.52s \\

   TSCSP & $48^{2}$(2304)& $\tau=h$ & 7.0552e-6 &  14&  68.378s \\

    GADI&  $48^{2}$(2304)& $\tau=h$ &7.5303e-6 & 7 &  36.525s\\

    GADI& $48^{2}$(2304)& $\tau=500h$ &4.4872e-6& 7& 101.35s\\
   \hline
 \end{tabular}
 \centering
 \end{table}

\begin{figure}[!htbp]

  \center
  \includegraphics[width=12.5cm,height=7cm] {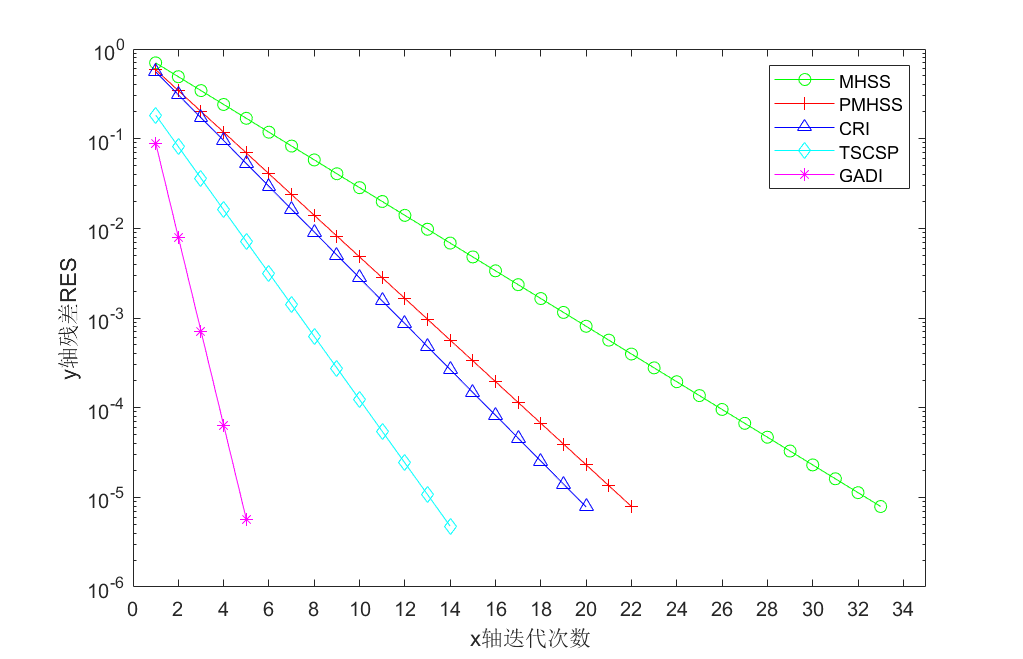}\\
  \caption{\label{1}$n=1024$\label{fig1}}
\end{figure}

In selecting the optimal parameters, we employed an experimental approach. In the iterative process of the GADI method, we first fixed the value of parameter $\omega$ and then explored the relationship between the number of iterations and the parameter $\alpha$, Finally, we found a relatively good parameter $\alpha_{opt}$.\\

\textbf{Example 2.4.2}\cite{ref12}\ \ We consider the complex Helmholtz equation $$-\triangle u+\sigma_{1}u+i\sigma_{2}u=f,$$
where $\sigma_{1}$ and $\sigma_{2}$ is a real coefficient function, and $u$ satisfies Dirichlet boundary conditions in the domain $D=[0,1]\times[0,1]$. We used finite differences on an $m\times m$ grid with grid size $h=\frac{1}{m+1}$ to discretize the problem. This leads to a linear system of equations
$$((K+\sigma_{1}I)+i\sigma_{2}I)x=b,$$
where $K=I\otimes V_{m}+V_{m}\otimes I$ is a discrete difference centered around $-\triangle$, $V_{m}=h^{-2} \text{tridiag}(-1,2,-1)\in\mathbb R^{m\times m}$. The right-hand vector $b=(1+i)A\bm{1}$, where$\bm{1}$ is a vector with all elements equal to 11. Additionally, before solving the system, we normalized the coefficient matrix and multiplied the right-hand side vector by $h^{2}$. For the numerical tests, we set $\sigma_{1}=\sigma_{2}=100$.

The table below presents the iteration counts (IT), iteration time (CPU), and residual norms (RES) for the MHSS, PMHSS, CRI, TSCSP, and GADI methods in Example 2.4.2. From Table \ref{tab2}, it is evident that the GADI method clearly outperforms the MHSS, PMHSS, CRI, and TSCSP methods. In Fig. \ref{fig2}, we can observe the variation of the residual for these methods with iteration steps, specifically for $m=32(n=1024)$. Among them, the GADI method exhibits the fastest decrease in residual, further demonstrating its superior convergence rate.
\begin{table}[!htbp]
\caption{Numerical results for Example 2.4.2\label{tab2}}
\begin{tabular}{lllll}
   \hline
   algorithm &$m^{2}$ & RES & IT &CPU\\
 \hline
   MHSS& $8^{2}$(64) & 8.5494e-6 & 23 & 0.052s  \\

    PMHSS& $8^{2}$(64) &9.6066e-6 & 18 & 0.045s   \\

    CRI & $8^{2}$(64) & 8.0214e-6 & 19& 0.0586s \\

   TSCSP&$8^{2}$(64)& 2.7332e-6 & 5 & 0.0226s  \\

    GADI& $8^{2}$(64)& 7.6366e-6 & 4 & 0.0208s   \\
    \hline
   MHSS& $16^{2}$(256) & 9.7665e-6 & 28 &0.7329s  \\

    PMHSS& $16^{2}$(256) &9.579e-6 & 21 & 0.649s   \\

    CRI & $16^{2}$(256) & 5.7801e-6& 16& 0.608s \\

   TSCSP&$16^{2}$(256)& 9.6159e-6 & 6 & 0.3527s  \\

    GADI& $16^{2}$(256)& 4.2332e-7 & 4 & 0.162s   \\
    \hline
   MHSS& $24^{2}$(576) & 8.4213e-6& 32& 4.099s  \\

    PMHSS& $24^{2}$(576) &8.8075e-6 & 22 & 3.0974s   \\

    CRI & $24^{2}$(576) & 5.5243e-6 & 17&3.005s \\

   TSCSP&$24^{2}$(576)& 7.1973e-6 & 6 & 1.7102s \\

    GADI& $24^{2}$(576)&1.5259e-6& 4 &  0.93s  \\
   \hline
   MHSS& $32^{2}$(1024) & 8.5731e-6 & 37 & 18.44s  \\

    PMHSS& $32^{2}$(1024) &6.5034e-6 & 26 &14.765s   \\

    CRI & $32^{2}$(1024) & 6.3379e-6 & 16& 12.015s \\

   TSCSP&$32^{2}$(1024)& 1.747e-6 & 6 &  7.0458s \\

    GADI& $32^{2}$(1024)& 1.7169e-6& 4 & 4.512s   \\
    \hline
   MHSS& $48^{2}$(2304) & 8.6389e-6 & 44 & 357.68s  \\

    PMHSS& $48^{2}$(2304) &9.1479e-6 & 32 & 200.31s   \\

    CRI & $48^{2}$(2304) & 9.8633e-6 & 17& 152.35s \\

   TSCSP&$48^{2}$(2304)& 1.6096e-6 & 6 & 66.866s \\

    GADI& $48^{2}$(2304)& 9.9397e-7 & 5 &  57.7s  \\
   \hline
 \end{tabular}
 \centering
 \end{table}
\begin{figure}[H]

  \center
  \includegraphics[width=12cm,height=7cm] {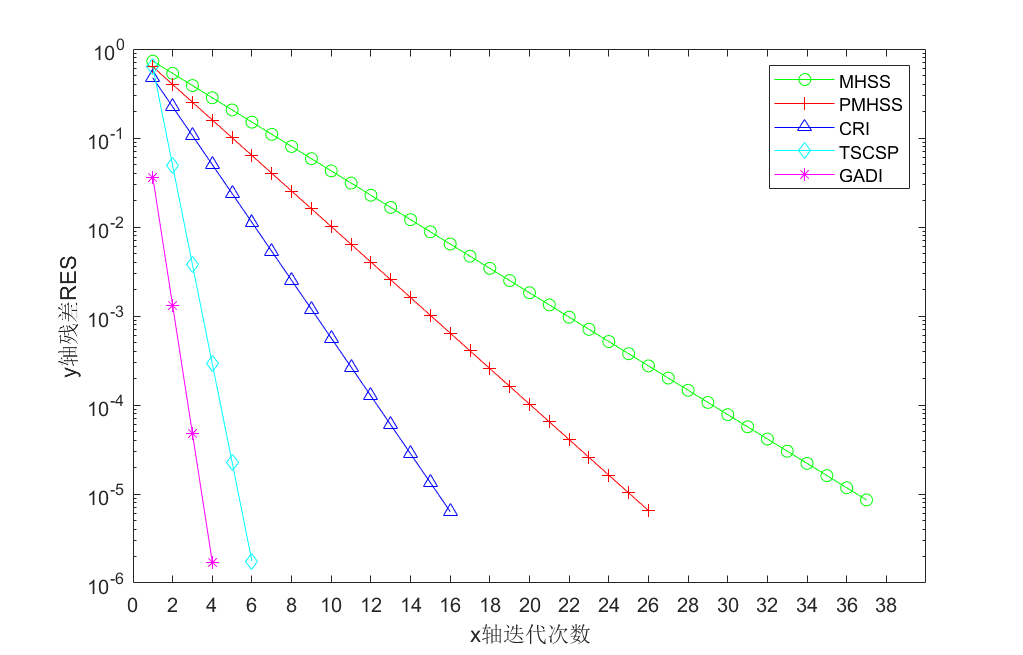}\\
  \caption{\label{1}$n=1024$ \label{fig2}}
\end{figure}
 
\section{GADI for solving Lyapunov equations with complex coefficients}
\label{sec:solve Lyapunov equation}
\qquad As an application of the GADI method for solving complex symmetric linear systems, we consider the Lyapunov matrix equation
\begin{equation}\label{eq12}
A^{\ast}X+XA=Q,\ \ \ A,\ Q,\ X\in\mathbb C^{n\times n},
\end{equation}
where $A, Q$ are known matrices, and $X$ is unknown matrices. We assume that matrix $A$ is sparse and has eigenvalues with positive real parts, i.e., $Re(\lambda_{i}(A))>0\ \ \forall i$. Additionally, $A$ can be decomposed as $A=W+iT$, where $W\in\mathbb R^{n\times n}$ is a symmetric positive-definite matrix and $T\in\mathbb R^{n\times n}$ is a symmetric semi positive-definite matrix. The matrix $Q$ is also symmetric positive definite, where $Q=C^{T}C$ and $C\in\mathbb C^{p\times n}$, with rank$(C)$=$p\ll n$.

Using the Kronecker product, we can transform Equation \eqref{eq12} into a linear system defined on an $n^{2}$ dimensional vector space. Here, we introduce vectorization operator vec: $\mathbb C^{m\times n}\rightarrow\mathbb C^{mn}$:
$$\text{vec}(X_{1})=(x_{1}^{T},x_{2}^{T},\cdots,x_{n}^{T})^{T},\ \ X_{1}=[x_{1},x_{2},\cdots,x_{n}]\in\mathbb C^{m\times n}.$$
For any $X_{1}\in\mathbb C^{m\times n},\ x=\text{vec}(X_{1})$ is a vector in $\mathbb C^{mn}$, so we refer to this operator as the vectorization operator.

\begin{lemma}\label{lemma3}\cite{ref20}
Let $A_{1}\in\mathbb C^{l\times m},\ X_{1}\in\mathbb C^{m\times n},\ B_{1}\in\mathbb C^{n\times k}$, then $$\text{vec}(A_{1}X_{1}B_{1})=(B_{1}^{T}\otimes A_{1})\text{vec}(X_{1}).$$
\end{lemma}
\begin{lemma}\label{lemma4}\cite{ref20}
Let $A_{1}\in\mathbb C^{m\times m}$ have eigenvalues $\lambda_{1},\cdots,\lambda_{m}$, and let  $B_{1}\in\mathbb C^{n\times n}$ have eigenvalues $\mu_{1},\cdots,\mu_{n}$. Then $$\lambda(I_{n}\otimes A_{1}-B^{T}_{1}\otimes I_{m})=\{\lambda_{i}-\mu_{j}|i=1,2,\cdots,m;\ j=1,2,\cdots,n\}.$$
\end{lemma}
By Lemma \ref{lemma3}, we can rewrite \eqref{eq12} as
\begin{equation}\label{eq13}
(A^{\ast}\otimes I+I\otimes A^{T})\text{vec}(X)=\text{vec}(Q).
\end{equation}
Thus, we have $$\widetilde{A}x=q,$$
where $$\widetilde{A}=A^{\ast}\otimes I+I\otimes A^{T},\ \ x=\text{vec}(X),\ \ q=\text{vec}(Q).$$
Since $W$ is a real symmetric positive definite matrix and $T$ is a real symmetric positive semi-definite matrix, we have $W^{T}=W, T^{T}=T$. As $A$ can be decomposed as $A=W+iT$, then \eqref{eq13} is equivalent to
\begin{equation}\label{eq14}
[(W\otimes I+I\otimes W)+i(I\otimes T-T\otimes I)]\text{vec}(X)=\text{vec}(Q).
\end{equation}
Therefore,we can obtain
\begin{equation}\label{eq15}
\widetilde{A}x=(\widetilde{W}+i\widetilde{T})x=q,
\end{equation}
where $$\widetilde{W}=W\otimes I+I\otimes W,\ \ \widetilde{T}=I\otimes T-T\otimes I.$$

Next, we apply the GADI iterative format to \eqref{eq15}, and there is
\begin{equation}\label{eq16}
\left\{
\begin{array}{ll}
(\alpha I_{n^{2}}+\widetilde{W})x^{(k+\frac{1}{2})}=(\alpha I_{n^{2}}-i\widetilde{T})x^{(k)}+q,\\
(\alpha I_{n^{2}}+i\widetilde{T})x^{(k+1)}=(i\widetilde{T}-(1-\omega)\alpha I_{n^{2}})x^{(k)}+(2-\omega)\alpha x^{(k+\frac{1}{2})}.
\end{array}
\right.
\end{equation}

\textbf{Example 3.1}\ \ We consider the coefficient matrix of the Lyapunov equation \eqref{eq12} to have the following form:$$A=W+iT=(M+2tN+\frac{100}{(n+1)^{2}}I)+i(M+2tN-\frac{100}{(n+1)^{2}}I),\ \ \ Q=C^{T}C.$$
Here, $C=\left(
                              \begin{array}{ccc}
                                1 ,& \cdots, & 1 \\
                              \end{array}
                            \right)_{1\times n}$, $t$ is a control parameter, and $M,\ N\in\mathbb R^{n\times n}$ is a tridiagonal matrix, where
$$
M=\left(
\begin{array}{ccccc}
    2   &  -1   &  0   &  \cdots   &  0\\
     -1 &  2   &  -1  &   \cdots   &  0\\
     \vdots   &  \ddots  &  \ddots  &   \ddots   &  \vdots\\
     0   &  \cdots   &  -1  &   2   &  -1\\
     0   & \cdots   &  0  &   -1  &   2\\
  \end{array}
\right)_{n\times n},\ \ \ N=\left(
\begin{array}{ccccc}
    0   &  0.5   &  0   &  \cdots   &  0\\
     0.5 &  0   &  0.5  &   \cdots   &  0\\
     \vdots   &  \ddots  &  \ddots  &   \ddots   &  \vdots\\
     0   &  \cdots   &  0.5  &   0   &  0.5\\
     0   & \cdots   &  0  &   0.5  &   0\\
  \end{array}
\right)_{n\times n}.
$$

After transforming \eqref{eq12} into \eqref{eq15}, we set $t$ to be 0.01 and 0.1, and then apply the HSS method and the GADI method \eqref{eq16} for solving. The iteration methods start from the initial value $x^{(0)}=0$, and once the current residual norm satisfies$$\frac{\|q-\widetilde{A}x^{(k)}\|_{2}}{\|q-\widetilde{A}x^{(0)}\|_{2}}\leq10^{-6},$$ or $\frac{\|R(X^{(k)})\|_{2}}{\|R(X^{(0)})\|_{2}}\leq10^{-6}$, the iteration stops, where $R(X^{(k)})=Q-A^{T}X^{(k)}-X^{(k)}A.$\\

In the process of solving the continuous Lyapunov equation, the HSS and GADI methods do not have a theoretical basis to estimate the relatively optimal splitting parameters. We denote $\gamma_{min},\gamma_{max}$ as the minimum and maximum eigenvalues of $\widetilde{W}=W^{T}\otimes I+I\otimes W^{T}$, respectively. When the matrix dimension $n=16$, by changing the parameter $\omega$, we obtained the relevant numerical results as shown in Table \ref{tab3}. In the GADI method, we set the parameter $\alpha=\sqrt{\gamma_ {min}\gamma_{max}}$  and tried different cases where $ \omega$ takes the values 0.01, 0.1, 0, 0.5, 1, 1.5. From the table data, it can be observed that under different control parameters $t$, the number of iterations and time decrease as $t$ increases. Under the same control parameter, when $\omega<0.5$, the number of iteration steps is relatively less.

\begin{table}[!htbp]

\caption{Numerical results for Example 3.1\label{tab3}}
\begin{tabular}{llllll}
   \hline
   \multicolumn{6}{c}{algorithm GADI}\\
     \hline
    $n$ &control parameters ($t$)& $(\alpha,\ \omega)$ &RES & IT & CPU \\
   \hline
    16&$t=0.01$& (2.6198,0.01) &5.749e-06& 19 & 0.40s   \\

    16& $t=0.01$&(2.6198,0.1) &5.7843e-06& 20 & 0.42s  \\

     16& $t=0.01$& (2.6198,0) &5.3789e-06& 19 & 0.41s  \\

    16&$t=0.01$& (2.6198,0.5) &9.8617e-06& 25 &0.53s  \\

     16& $t=0.01$&(2.6198,1) &8.4038e-06& 40 &0.84s  \\

     16& $t=0.01$&(2.6198,1.5) &8.9073e-06& 84 &1.75s  \\
   \hline
    16& $t=0.1$&(3.081,0.01) &6.9929e-06& 15 & 0.35s   \\

    16&$t=0.1$& (3.081,0.1) &7.2687e-06& 16 & 0.36s  \\

    16&$t=0.1$  & (3.081,0) &6.4057e-06& 15 & 0.35s  \\

    16&$t=0.1$& (3.081,0.5) &7.7465e-06& 22 &0.49s  \\

     16& $t=0.1$&(3.081,1) &8.4099e-06& 36 &0.75s \\

     16& $t=0.1$&(3.081,1.5) &9.9062e-06& 77 &1.58s  \\
   \hline
 \end{tabular}
 \centering
 \end{table}

 We obtained the computational results of the HSS and GADI methods by gradually increasing the matrix dimension and trying different control parameters $t$, as shown in Table \ref{tab4}. These numerical results demonstrate that the HSS and GADI methods can effectively solve the continuous Lyapunov equation \eqref{eq12} with appropriate splitting parameters. We observed that the number of iterations and time obtained by both methods increase with the matrix dimension, but the GADI method is more efficient than the HSS method. For instance, when $n=32$, Fig. \ref{fig3} and Fig. \ref{fig4} illustrate the iteration steps and corresponding residuals of the HSS and GADI methods. To highlight the advantage of the GADI method, we showcase the time savings in Fig. \ref{fig5} and Fig. \ref{fig6}, where it is evident that the GADI method significantly reduces the computational time compared to the HSS method.

 \begin{table}[!htbp]
\caption{Numerical results for Example 3.1 \label{tab4}}
\begin{tabular}{llllll}
   \hline
    algorithm &$n$& control parameters ($t$)  & RES &IT &CPU  \\
    \hline
    HSS&8& $t=0.01$  & 8.4167e-06&13 & 0.03s  \\

    HSS& 8& $t=0.1$  & 6.493e-06 & 12 & 0.03s  \\

   GADI & 8& $t=0.01$ & 8.9841e-06& 10& 0.02s \\

   GADI & 8& $t=0.1$ & 2.872e-06 &  10&  0.02s \\
    \hline
    HSS&16& $t=0.01$  & 8.6899e-06&25 & 0.62s  \\

    HSS& 16& $t=0.1$  & 9.1129e-06 & 19 & 0.52s  \\

   GADI & 16& $t=0.01$ & 5.3789e-06& 19& 0.41s \\

   GADI & 16& $t=0.1$ & 6.4057e-06 &  15&  0.32s \\
    \hline
    HSS&24& $t=0.01$  & 8.3902e-06&36 & 5.36s  \\

    HSS& 24& $t=0.1$  & 5.9927e-06 & 23 & 3.73s  \\

   GADI & 24& $t=0.01$ &8.4557e-06& 26& 2.97s \\

   GADI & 24& $t=0.1$ & 8.9194e-06 &  18&  2.08s \\
    \hline
    HSS&32& $t=0.01$  & 9.7102e-06&45 & 40.23s  \\

    HSS& 32& $t=0.1$  & 6.7267e-06 & 24 & 21.64s  \\

   GADI & 32& $t=0.01$ & 8.4981e-06& 33& 17.33s \\

   GADI & 32& $t=0.1$ & 9.2584e-06 &  20&  10.65s \\
    \hline
    HSS&48& $t=0.01$  & 8.6595e-06&60 & 548.93s  \\

    HSS& 48& $t=0.1$  & 9.5955e-06 & 24 & 231.9s  \\

   GADI & 48& $t=0.01$ & 7.7812e-06& 45& 314.18s \\

   GADI & 48& $t=0.1$ & 9.7814e-06 &  22&  151.69s \\
   \hline
 \end{tabular}
 \centering
 \end{table}

\begin{figure}[H]
\centering
    \begin{minipage}[t]{0.49\textwidth}
        \centering
        \includegraphics[width=1.1\textwidth]{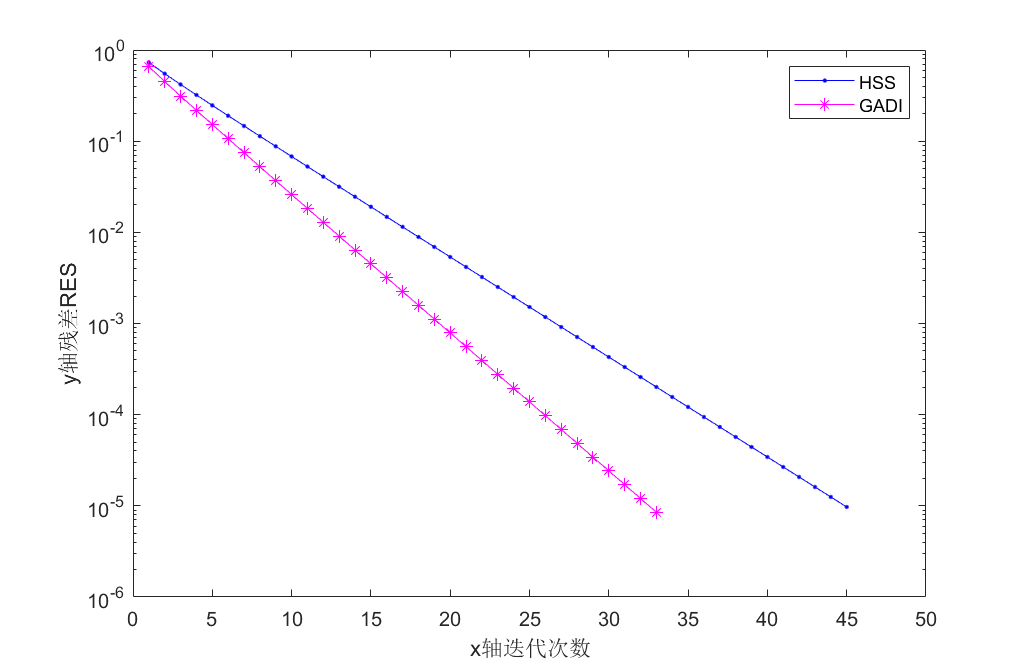}
        \caption{\label{1}$t=0.01$\label{fig3}}
        
    \end{minipage}
    \begin{minipage}[t]{0.49\textwidth}
        \centering
        \includegraphics[width=1.1\textwidth]{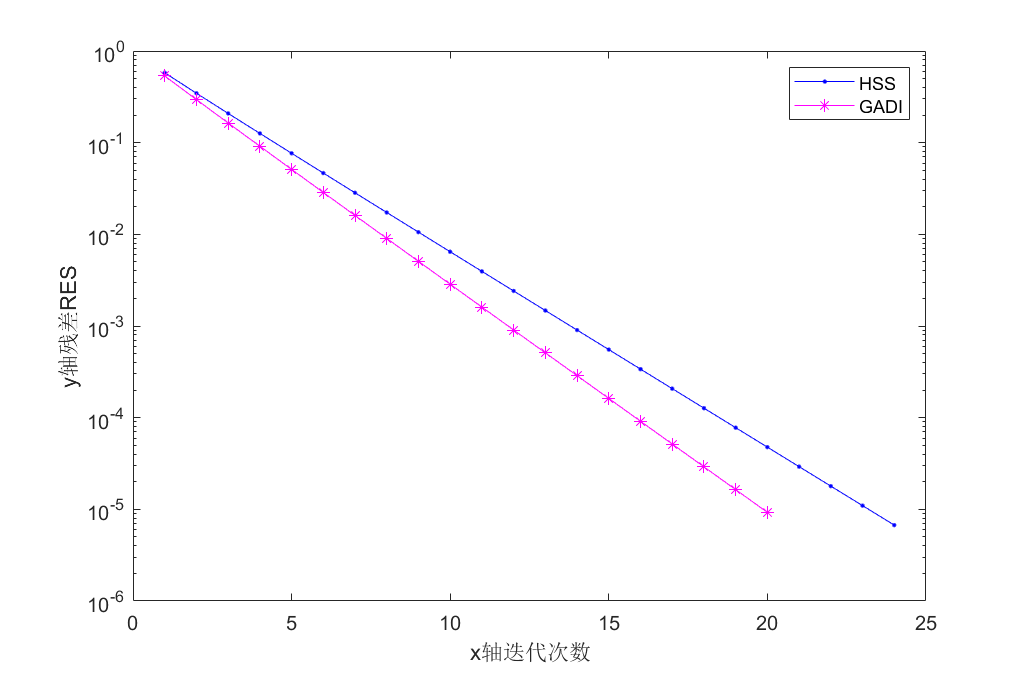}
        \caption{\label{1}$t=0.1$\label{fig4}}
        \end{minipage}
\end{figure}
\begin{figure}[H]
\centering
    \begin{minipage}[t]{0.49\textwidth}
        \centering
        \includegraphics[width=1.1\textwidth]{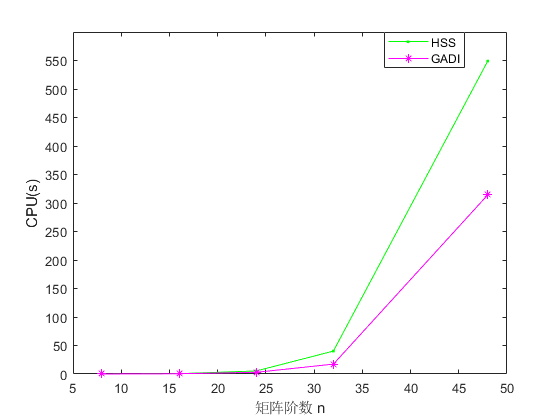}
        \caption{\label{1}$t=0.01$\label{fig5}}
       
    \end{minipage}
    \begin{minipage}[t]{0.49\textwidth}
        \centering
        \includegraphics[width=1.1\textwidth]{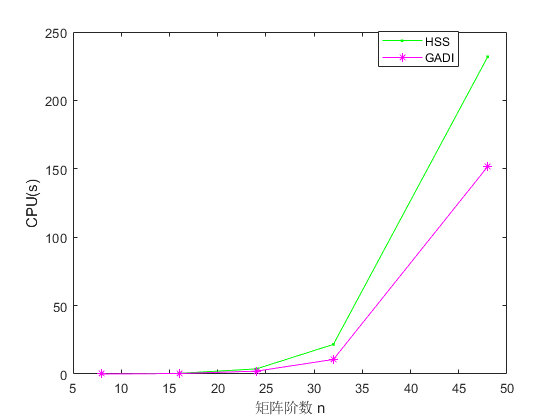}
        \caption{\label{1}$t=0.1$\label{fig6}}
        \end{minipage}
\end{figure}
\section{GADI for solving the Riccati equation with complex coefficients}
\label{sec:solve Riccati equation}
\qquad We consider the continuous time algebraic Riccati equation (CARE):
\begin{equation}\label{eq17}
  A^{\ast}X+XA+Q-XGX=0,
\end{equation}
where $A,\ Q,\ G\in\mathbb C^{n\times n},\ Q^{\ast}=Q,\ K^{\ast}=K$ are known matrix, and $X\in\mathbb C^{n\times n}$ is an unknown matrix. Here, $A$ can be decomposed as $A=W+iT$, where $W\in\mathbb R^{n\times n}$ is a symmetric positive definite matrix, and $T\in\mathbb R^{n\times n}$ is a symmetric positive semi-definite matrix.

Next, we use the Newton’s method to solve equation \eqref{eq17} and obtain the following equation:
\begin{equation}\label{eq18}
  (A-GX_{k})^{\ast}X_{k+1}+X_{k+1}(A-GX_{k})+X_{k}GX_{k}+Q=0.
\end{equation}
Here, we let $$A_{k}=A-GX_{k},\ Q_{k}=-X_{k}GX_{k}-Q,$$then we can obtain the Lyapunov equation equivalent to \eqref{eq18}
\begin{equation}\label{eq19}
  A_{k}^{\ast}X_{k+1}+X_{k+1}A_{k}=Q_{k}.
\end{equation}
According to Lemma \ref{lemma3}, we can rewrite \eqref{eq19} as
\begin{equation}\label{eq20}
(A_{k}^{\ast}\otimes I+I\otimes A_{k}^{T})\text{vec}(X_{k+1})=\text{vec}(Q_{k}).
\end{equation}
Therefore,$$\widetilde{A}_{k}x_{k+1}=q_{k},$$
where $$\widetilde{A}_{k}=A_{k}^{\ast}\otimes I+I\otimes A_{k}^{T},\ \ x_{k+1}=\text{vec}(X_{k+1}),\ \ q_{k}=\text{vec}(Q_{k}).$$
Since $W$ is a real symmetric positive definite matrix and $T$ is a real symmetric positive semi-definite matrix, then we have $W^{T}=W, T^{T}=T$. Due to $A=W+iT$, then $A_{k}=W+iT-GX_{k}$. Substituting this into \eqref{eq20}, we get
\begin{equation}\label{eq21}
[W\otimes I+I\otimes W+i(I\otimes T-T\otimes I)-(X_{k}G\otimes I+I\otimes X_{k}G)]x_{k+1}=q_{k}.
\end{equation}
That is,\begin{equation}\label{eq22}
\widetilde{A}_{k}x_{k+1}=(W_{k}+iT_{k}-G_{k})x_{k+1}=q_{k},
\end{equation}
where$$W_{k}=W\otimes I+I\otimes W,\ \ \ T_{k}=I\otimes T-T\otimes I,\ \ \ G_{k}=X_{k}G\otimes I+I\otimes X_{k}G.$$

Next, we apply the GADI iteration scheme to \eqref{eq22} and obtain
\begin{equation}\label{eq23}
\left\{
\begin{array}{ll}
(\alpha_{k}I_{n^{2}}+W_{k})x_{k+1}^{(l+\frac{1}{2})}=(\alpha_{k}I_{n^{2}}-iT_{k}+G_{k})x_{k+1}^{(l)}+q_{k},\\
(\alpha_{k}I_{n^{2}}+iT_{k}-G_{k})x_{k+1}^{(l+1)}=(iT_{k}-G_{k}-(1-\omega_{k})\alpha_{k}I_{n^{2}})x_{k+1}^{(l)}+(2-\omega_{k})\alpha_{k}x_{k+1}^{(l+\frac{1}{2})},
\end{array}
\right.
\end{equation}
where $l=0,1,\cdots,\ x_{k+1}^{(0)}=x_{k}, \alpha_{k}>0$ and $\omega_{k}\in[0,2)$.

Combining Newton's method and GADI iterative method, we get Newton GADI algorithm, which is summarized in algorithm \ref{alg2} .
\floatname{algorithm}{Algorithm}
\renewcommand{\algorithmicrequire}{\textbf{Input:}}
\renewcommand{\algorithmicensure}{\textbf{Output:}}
\begin{algorithm}
  \caption{Newton GADI algorithm for solving the Riccati equation \eqref{eq17} with complex coefficients}
  \label{alg2}
  \begin{algorithmic}[1]
    \REQUIRE Matrix $W, T, Q, G $, where $A=W+iT$, parameter $\alpha_{k}, \omega_{k}$, integers $k_{max}>1$ and $l_{max}>1$, outer and inner iteration error tolerances $\varepsilon_{out}$ and $\varepsilon_{inn}$;
    \ENSURE Approximation $X_{k+1}\approx X$, where $X$ is a solution of Riccati equation $A^{\ast}X+XA+Q-XGX=0$.
    \STATE Calculae $\beta=1+\|A\|_{\infty}, B=A+\beta I$, and obtain the initial matrix $X_{0}=X^{-1}$ by solving the Lyapunov equation $B^{\ast}X+XB-2Q=0$;
    \STATE Let $X_{k+1}^{(0)}=X_{0},$ and calculate $Q_{k}=-X_{k}GX_{k}-Q;$
    \STATE Calculate $$W_{k}=W\otimes I+I\otimes W,\ \ \ T_{k}=I\otimes T-T\otimes I,$$ $$G_{k}=X_{k}G\otimes I+I\otimes X_{k}G,\ \ \ \widetilde{A}_{k}=W_{k}+iT_{k}-G_{k};$$
    \STATE Calculate $x_{0}=\text{vec}(X_{0}),\ \ q_{k}=\text{vec}(Q_{k})$,\ obtain $x_{0}$ and $q_{k}$, let $x_{k+1}^{(0)}=x_{0};$
    \STATE Solve GADI iteration framework \eqref{eq23} to obtain $x_{k+1}=x_{k+1}^{(l_{k})}$ such that $$\|\widetilde{A}_{k}x_{k+1}^{(l_{k})}-q_{k}\|_{2}<\varepsilon_{inn},\ or\ l_{k}>l_{max};$$
    \STATE Calculate $x_{k+1}=\text{vec}(X_{k+1});$
    \STATE Calculate relative residuals $$\text{Res}(X_{k+1})=\frac{\|A^{\ast}X_{k+1}+X_{k+1}A+Q-X_{k+1}GX_{k+1}\|_{2}}{\|Q\|_{2}};$$
           \IF {$\text{Res}(X_{k+1})<\varepsilon_{out}$ or $k>k_{max}$}
               \STATE  Obtain an approximate solution $\widetilde{X}=X_{k+1}$ for the Riccati equation \eqref{eq17};
           \ELSE
               \STATE Let $k=k+1$ and return to step 2;
           \ENDIF
  \end{algorithmic}
\end{algorithm}

Next, we analyze the computational complexity of algorithm \ref{alg2}. This algorithm combines the external iterative Newton method with the internal iterative GADI method. The computational complexity of the external iteration mainly involves solving equation \eqref{eq19}, while the internal iteration's complexity mainly involves solving equation \eqref{eq23}. Here, the computations mainly involve matrix addition, subtraction, multiplication, inversion, as well as matrix flattening and vector operations. Since $W_{k},\ T_{k}\in\mathbb R^{n^{2}\times n^{2}},\ G_{k}\in\mathbb C^{n^{2}\times n^{2}},\ x_{k+1}\in\mathbb C^{n^{2}}$, then the computational complexity of the matrix $Q_{k}$ is $2n^{3}+n^{2}$. The computational complexity of $W_{k} , T_{k} , G_{k}$ are $2n^{4},\ 2n^{4},\ 2n^{4}+2n^{2}$, respectively, and the computational complexity of solving \eqref{eq22} with CG method is $O(n^{2})$. Therefore, the total computational complexity of algorithm 2 is $6n^{4}+2n^{3}+O(n^{2}).$
\subsection{Convergence analysis}
\qquad As the Newton iteration is quadratically convergent, to prove the convergence of Algorithm \ref{alg2}, we only need to demonstrate that the GADI iteration format \eqref{eq23} converges.

\begin{theorem}\label{theorem3}
For the CARE \eqref{eq17}, assuming that $(A, G)$ is stable and $(A, Q)$ is detectable, where $A=W+iT\in\mathbb C^{n\times n}$ is a non-singular matrix, with $W\in\mathbb R^{n\times n}$ being a symmetric positive definite matrix, and $T\in\mathbb R^{n\times n}$ being a symmetric positive semi-definite matrix. If the parameter $\alpha_{k}>0,\ \ 0\leq\omega_{k}<2$, then for any $k=0,1,2,\cdots$, the iteration sequence $\{x_{k+1}^{(l)}\}_{l=0}^{\infty}$ defined by \eqref{eq23} converges to $x_{k+1}$.
\end{theorem}
\begin{proof}
Based on equation \eqref{eq17}, we have the following formula:
$$x^{(l+\frac{1}{2})}_{k+1}=(\alpha_{k}I_{n^{2}}+W_{k})^{-1}[(\alpha_{k}I_{n^{2}}-iT_{k}+G_{k})x^{(l)}_{k+1}+q_{k}],$$
\begin{align*}
x^{(l+1)}_{k+1}&=(\alpha_{k}I_{n^{2}}+iT_{k}-G_{k})^{-1}[(iT_{k}-G_{k}-(1-\omega_{k})\alpha_{k}I_{n^{2}})x^{(l)}_{k+1}+(2-\omega_{k})\alpha_{k} x^{(l+\frac{1}{2})}_{k+1}]\\
&=(\alpha_{k}I+iT_{k}-G_{k})^{-1}[(iT_{k}-G_{k}-(1-\omega_{k})\alpha_{k}I_{n^{2}})+(2-\omega_{k})\alpha_{k}(\alpha_{k}I_{n^{2}}+W_{k})^{-1}(\alpha_{k}I_{n^{2}}\\
&-iT_{k}+G_{k})]x^{(l)}_{k+1}+(2-\omega_{k})\alpha_{k}(\alpha_{k}I_{n^{2}}+iT_{k}-G_{k})^{-1}(\alpha_{k}I_{n^{2}}+W_{k})^{-1}q_{k}\\
&=(\alpha_{k}I+iT_{k}-G_{k})^{-1}(\alpha_{k}I_{n^{2}}+W_{k})^{-1}[\alpha_{k}^{2}I_{n^{2}}+iW_{k}T_{k}-W_{k}G_{k}-(1-\omega_{k})\alpha_{k}A_{k}]x^{(l)}_{k+1}\\
&+(2-\omega_{k})\alpha_{k}(\alpha_{k}I_{n^{2}}+iT_{k}-G_{k})^{-1}(\alpha_{k}I_{n^{2}}+W_{k})^{-1}q_{k}.\\
\end{align*}
Here, we define $$M_{k}(\alpha_{k},\omega_{k})=(\alpha_{k}I+iT_{k}-G_{k})^{-1}(\alpha_{k}I_{n^{2}}+W_{k})^{-1}[\alpha_{k}^{2}I_{n^{2}}+iW_{k}T_{k}-W_{k}G_{k}-(1-\omega_{k})\alpha_{k}A_{k}],$$
$$N_{k}(\alpha_{k},\omega_{k})=(2-\omega_{k})\alpha_{k}(\alpha_{k}I_{n^{2}}+iT_{k}-G_{k})^{-1}(\alpha_{k}I_{n^{2}}+W_{k})^{-1}q_{k},$$
then we have $$x^{(l+1)}_{k+1}=M_{k}(\alpha_{k},\omega_{k})x^{(l)}_{k}+N_{k}(\alpha_{k},\omega_{k}).$$
Next, it is necessary to prove that for$\alpha_{k}>0,\ \ 0\leq\omega_{k}<2$, there is $\rho(M_{k}(\alpha_{k},\omega_{k}))<1$. Since
\begin{align*}
2M_{k}(\alpha_{k},\omega_{k})&=2(\alpha_{k}I+iT_{k}-G_{k})^{-1}(\alpha_{k}I_{n^{2}}+W_{k})^{-1}[\alpha_{k}^{2}I_{n^{2}}+iW_{k}T_{k}-W_{k}G_{k}-(1-\omega_{k})\alpha_{k}A_{k}]\\
&=(2-\omega_{k})(\alpha_{k}I+iT_{k}-G_{k})^{-1}(\alpha_{k}I_{n^{2}}+W_{k})^{-1}(\alpha_{k}I_{n^{2}}-W_{k})(\alpha_{k}I_{n^{2}}-iT_{k}+G_{k})\\
&+\omega_{k}(\alpha_{k}I+iT_{k}-G_{k})^{-1}(\alpha_{k}I_{n^{2}}+W_{k})^{-1}(\alpha_{k}I_{n^{2}}+W_{k})(\alpha_{k}I+iT_{k}-G_{k})\\
&=(2-\omega_{k})T_{k}(\alpha_{k})+\omega_{k}I_{n^{2}},
\end{align*}
where $$T_{k}(\alpha_{k})=(\alpha_{k}I+iT_{k}-G_{k})^{-1}(\alpha_{k}I_{n^{2}}+W_{k})^{-1}(\alpha_{k}I_{n^{2}}-W_{k})(\alpha_{k}I_{n^{2}}-iT_{k}+G_{k}),$$
therefore, we can get $$M_{k}(\alpha_{k},\omega_{k})=\frac{1}{2}[(2-\omega_{k})T_{k}(\alpha_{k})+\omega_{k}I_{n^{2}}].$$
Due to $$\lambda_{i}(M_{k}(\alpha_{k},\omega_{k}))=\frac{1}{2}[(2-\omega_{k})\lambda_{i}(T_{k}(\alpha_{k}))+\omega_{k}]\ (i=1,2,\cdots n),$$
then $$\rho(M_{k}(\alpha_{k},\omega_{k}))\leq\frac{1}{2}[(2-\omega_{k})\rho(T_{k}(\alpha_{k}))+\omega_{k}].$$
Let $$\widetilde{T}_{k}(\alpha_{k})=(\alpha_{k}I_{n^{2}}+W_{k})^{-1}(\alpha_{k}I_{n^{2}}-W_{k})(\alpha_{k}I_{n^{2}}-iT_{k}+G_{k})(\alpha_{k}I+iT_{k}-G_{k})^{-1},$$ it can be seen that $T_{k}(\alpha_{k})$  is similar to $\widetilde{T}_{k}(\alpha_{k})$  through the matrix $\alpha_{k}I+iT_{k}-G_{k}$. Therefore, we have
\begin{align*}
\rho(T_{k}(\alpha_{k}))&\leq\parallel(\alpha_{k}I_{n^{2}}+W_{k})^{-1}(\alpha_{k}I_{n^{2}}-W_{k})(\alpha_{k}I_{n^{2}}-iT_{k}+G_{k})(\alpha_{k}I+iT_{k}-G_{k})^{-1}\parallel_{2}\\
&\leq\parallel(\alpha_{k}I_{n^{2}}+W_{k})^{-1}(\alpha_{k}I_{n^{2}}-W_{k})\parallel_{2}\parallel(\alpha_{k}I_{n^{2}}-iT_{k}+G_{k})(\alpha_{k}I+iT_{k}-G_{k})^{-1}\parallel_{2}\\
&=\parallel W_{k}^{L}\parallel_{2}\parallel T_{k}^{R}\parallel_{2},
\end{align*}
where $$W_{k}^{L}=(\alpha_{k}I_{n^{2}}+W_{k})^{-1}(\alpha_{k}I_{n^{2}}-W_{k}),\ \ T_{k}^{R}=(\alpha_{k}I_{n^{2}}-iT_{k}+G_{k})(\alpha_{k}I+iT_{k}-G_{k})^{-1}.$$
For $y\in\mathbb C^{n^{2}\times1}$, we get
\begin{align*}
\parallel W_{k}^{L}\parallel_{2}^{2}&=\max_{\parallel y\parallel_{2}=1}\frac{\parallel(\alpha_{k}I_{n^{2}}-W_{k})y\parallel_{2}^{2}}{\parallel(\alpha_{k}I_{n^{2}}+W_{k})y\parallel_{2}^{2}}\\
&=\max_{\parallel y\parallel_{2}=1}\frac{\parallel W_{k}y\parallel_{2}^{2}-\alpha_{k}y^{T}(W_{k}^{T}+W_{k})y+\alpha_{k}^{2}}{\parallel W_{k}y\parallel_{2}^{2}+\alpha_{k}y^{T}(W_{k}^{T}+W_{k})y+\alpha_{k}^{2}}\\
&\leq\max_{\parallel y\parallel_{2}=1}\frac{\parallel W_{k}y\parallel_{2}^{2}-2\alpha_{k}\min Re(\lambda(W_{k}))+\alpha_{k}^{2}}{\parallel W_{k}y\parallel_{2}^{2}+2\alpha_{k}\min Re(\lambda(W_{k}))+\alpha_{k}^{2}}\\
&\leq\frac{\parallel W_{k}\parallel_{2}^{2}-2\alpha_{k}\min Re(\lambda(W_{k}))+\alpha_{k}^{2}}{\parallel W_{k}\parallel_{2}^{2}+2\alpha\min Re(\lambda(W_{k}))+\alpha_{k}^{2}}.
\end{align*}
There is also a conclusion
$$\parallel T_{k}^{R}\parallel_{2}^{2}\leq\frac{\parallel iT_{k}-G_{k}\parallel_{2}^{2}-2\alpha_{k}\min Re(\lambda(iT_{k}-G_{k}))+\alpha_{k}^{2}}{\parallel iT_{k}-G_{k}\parallel_{2}^{2}+2\alpha_{k}\min Re(\lambda(iT_{k}-G_{k}))+\alpha_{k}^{2}}.$$
Since $\alpha_{k}>0$, we have $$\parallel W_{k}^{L}\parallel_{2}<1,\ \ \parallel T_{k}^{R}\parallel_{2}<1,$$
thus implying $\rho(T_{k}(\alpha_{k}))<1$. Furthermore, as $0\leq\omega_{k}<2$, it follows that 
$$\rho(M_{k}(\alpha_{k},\omega_{k}))\leq\frac{1}{2}[(2-\omega_{k})\rho(T_{k}(\alpha_{k}))+\omega_{k}]<1.$$
Therefore, the iterative format in equation \eqref{eq23} converges.
\end{proof}
\subsection{Numerical experiments}
\qquad\textbf{Example 4.2.1\ \ }We consider the coefficient matrices of the Riccati equation \eqref{eq17} as $$A=W+iT,\ Q=C^{T}C,\ G=10^{-1}I_{n},$$ where $I_{n}$ is the $n$-order identity matrix, $C=\left(
                              \begin{array}{ccc}
                                1 ,& \cdots, & 1 \\
                              \end{array}
                            \right)_{1\times n}$, and $W,\ T\in\mathbb R^{n\times n}$ are tridiagonal matrices defined as follows:$$
W=\left(
\begin{array}{ccccc}
    2   &  -1   &  0   &  \cdots   &  0\\
     -1 &  2   &  -1  &   \cdots   &  0\\
     \vdots   &  \ddots  &  \ddots  &   \ddots   &  \vdots\\
     0   &  \cdots   &  -1  &   2   &  -1\\
     0   & \cdots   &  0  &   -1  &   2\\
  \end{array}
\right)_{n\times n},\ \ \ N=\left(
\begin{array}{ccccc}
    0.5   &  0.1   &  0   &  \cdots   &  0\\
     0.1 &  0.5   &  0.1  &   \cdots   &  0\\
     \vdots   &  \ddots  &  \ddots  &   \ddots   &  \vdots\\
     0   &  \cdots   &  0.1  &   0.5   &  0.1\\
     0   & \cdots   &  0  &   0.1  &   0.5\\
  \end{array}
\right)_{n\times n}.$$

\begin{table}[!htbp]

\caption{Numerical results for Example 4.2.1\label{tab5}}
\begin{tabular}{lllll}
   \hline
   algorithm &$n$&Res & IT & CPU(s)  \\
   \hline
   Newton-GADI& 8 & 7.485e-06 & 33 & 0.06s    \\
   \hline
   Newton-GADI& 16 & 9.0282e-06 & 62 & 1.25s   \\
   \hline
   Newton-GADI& 24 & 9.3853e-06 &91 & 12.51s    \\
   \hline
   Newton-GADI& 32& 9.5209e-06 &120 & 64.16s    \\
   \hline
   Newton-GADI& 48  & 9.6239e-06&178 & 1140.9s   \\
   \hline
   Newton-GADI& 64& 9.6621e-06 &236 & 89662s    \\
   \hline
 \end{tabular}
 \centering
 \end{table}
 \begin{figure}[H]
  \center
  \includegraphics[width=13cm,height=8cm] {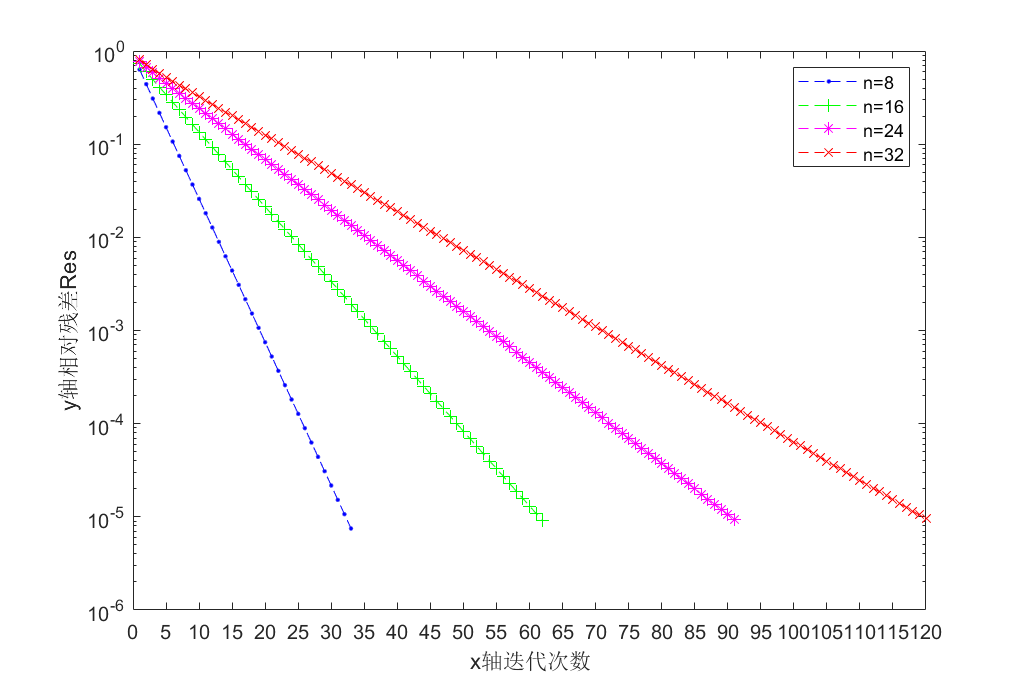}\\
  \caption{Residual norm for Example 4.2.1\label{fig7}}
\end{figure}
Next, we use the Newton-GADI method to compute this example and present the numerical results in Table \ref{tab5}. From the results, it can be observed that with the increase in matrix order, the number of iterations and computation time also increase. Therefore, this method is suitable for computing small-sized matrices, and for large-sized matrices, the computation speed becomes very slow. Fig. \ref{fig7} displays the iteration steps and residual norms for matrix orders of $n=8,16,24$. We can see that as the matrix order increases, the convergence rate of this algorithm decreases.
\section{Conclusions}
\label{sec:Conclusions}
\qquad This article proposes the GADI iteration method for solving large sparse complex symmetric linear systems. The method is applied to solve Lyapunov and Riccati equations with complex coefficients, utilizing the properties of the flattening operator and Kronecker product. Additionally, we analyze the convergence of this method and demonstrate its effectiveness through corresponding numerical results. However, this method also has certain limitations as it involves solving two sublinear systems during the iteration process, with the second sublinear system being a complex linear system that requires complex algorithms.  Additionally, the selection of parameters can also be further discussed.




\begin{thebibliography}{99}
\setlength{\itemsep}{2mm}
\bibitem{ref1}Arridge, S.R. Optical tomography in medical imaging. {\textit {Inverse Probl}}. 1999, 15, 41-93.
\bibitem{ref2} Feriani, A., Perotti, F., Simoncini, V. Iterative system solvers for the frequency analysis of linear
mechanical systems. {\textit {Comput. Methods Appl. Mech. Eng}}. 2000, 190, 1719-1739.
\bibitem{ref3}Dijk, W. V., Toyama, F. M. Accurate numerical solutions of the time-dependent
Schr¡§odinger equation. {\textit {Phys. Rev. E}}. 2007, 75.
\bibitem{ref4}Poirier, B. Efficient preconditioning scheme for block partitioned matrices with structured sparsity. {\textit {Numer.
Linear Algebra Appl.}} 2000, 7, 715-726.
\bibitem{ref5}Schmitt, D., Steffen, B., Weiland, T. 2D and 3D computations of lossy eigenvalue problems. {\textit {IEEE Trans. Magn.}} 1994, 30, 3578-3581.
\bibitem{ref6}Bertaccini, D. Efficient solvers for sequences of complex symmetric linear systems. {\textit {Electr. Trans. Numer.
Anal.}} 2004, 18, 49-64.
\bibitem{ref7}Bai, Z. Z., Golub, G. H., Ng, M. K. Hermitian and skew-Hermitian splitting methods for non-Hermitian
positive definite linear systems. {\textit {SIAM J Matrix Anal Appl}}. 2003, 24,  603-626.
\bibitem{ref8}Bai, Z. Z., Golub, G. H., Pan, J. Y. Preconditioned Hermitian and skew-Hermitian splitting methods for nonHermitian positive semidefinite linear systems. {\textit {Numer Math.}} 2004, 98, 1-32.
\bibitem{ref9}Bai, Z. Z., Golub, G. H. Accelerated Hermitian and skew-Hermitian splitting iteration methods for saddlepoint problems. {\textit {IMA J Numer Anal.}} 2007, 27, 1-23.
\bibitem{ref10}Bai, Z. Z., Benzi, M., Chen, F. Modified HSS iteration methods for a class of complex symmetric linear
systems. {\textit {Computing}}. 2010, 87, 93-111.
\bibitem{ref11}Bai, Z. Z., Benzi, M., Chen, F. On preconditioned MHSS iteration methods for complex symmetric linear
systems. {\textit {Numer. Algor.}} 2011, 56, 297-317.
\bibitem{ref12}Hezari, D., Edalatpour, V., Salkuyeh, DK. Preconditioned GSOR iterative method for a class of complex
symmetric system of linear equations. {\textit {Numer Linear Algebra Appl}}. 2015, 22, 761-776.
\bibitem{ref13}Wang, T., Zheng, Q. Q., Lu, L. Z. A new iteration method for a class of complex symmetric linear systems. {\textit {J Comput Appl Math.}} 2017, 325, 188-197.
\bibitem{ref14}Hezari, D., Salkuyeh, D. K., Edalatpour, V. A new iterative method for solving a class
of complex symmetric system of linear equathions. {\textit {Numer. Algor.}} 2016, 73, 927-955.
\bibitem{ref15}Salkuyeh, D. K. Two-step scale-splitting method for solving complex symmetric system
of linear equations. 2017, arXiv:1705.02468.
\bibitem{ref16}Wang, T., Zheng, Q. Q., Lu, L. Z. A new iteration method for a class of complex symmetric
linear systems. {\textit {J Comput Appl Math.}} 2017, http://dx.doi.org/10.1016/j.cam.2017.05.002.
\bibitem{ref17}Salkuyeh, D. K., Siahkolaei, T. S. Two-parameter TSCSP method for solving complex symmetric system
of linear equations. Calcolo, 2018, 55, 1-22.
\bibitem{ref18}Yang, A. L., Cao, Y., Wu, Y. J. Minimum residual Hermitian and skew-Hermitian splitting iteration method
for non-Hermitian positive definite linear systems. {\textit {BIT Numer. Math.}} 2019, 59, 299-319.
\bibitem{ref19}Jiang, K., Su, X. H., Zhang, J. A general alternating-direction implicit framework with
Gaussian process regression parameter prediction for large sparse linear systems. {\textit {SIAM Journal on Scientific Computing}}. 2021, https://arxiv.org/abs/2109.12249.
\bibitem{ref20}Xu, S. F. In {\textit {Matrix calculation in cybernetics}}, Beijing: Higher Education Press, 2011.

\end{thebibliography}

\end{document}